\documentclass[11pt]{amsart}
\usepackage{amsfonts}
\usepackage{amsmath}
\usepackage{amssymb}
\usepackage{graphicx,tikz}%
\usepackage{amsthm}

\usepackage[toc,page]{appendix}

\usepackage{latexsym}
\usepackage{hyperref}
\usepackage{enumerate}
\usepackage{graphicx}

\usepackage{textgreek}

\usepackage[utf8]{inputenc}
\hypersetup{
    colorlinks,
    citecolor=blue,
    filecolor=blue,
    linkcolor=blue,
    urlcolor=blue
}

\usepackage[all]{xy}
\providecommand{\U}[1]{\protect\rule{.1in}{.1in}}

\numberwithin{equation}{subsection}

\newtheorem{theorem}{Theorem}
\newtheorem*{theorem-intro}{Theorem}

\newtheorem{corollary}[equation]{Corollary}

\newtheorem{lemma}[equation]{Lemma}

\newtheorem{proposition}[equation]{Proposition}
\newtheorem{remark}[equation]{Remark}

\DeclareMathOperator{\Sp}{Sp}
\DeclareMathOperator{\SL}{SL}
\DeclareMathOperator{\GL}{GL}
\DeclareMathOperator{\GSp}{GSp}

\DeclareMathOperator{\Id}{I}
\DeclareMathOperator{\ind}{Ind}
\DeclareMathOperator{\cind}{c-Ind}
\DeclareMathOperator{\infl}{Inf}

\DeclareMathOperator{\Irr}{Irr}
\DeclareMathOperator{\Res}{Res}
\DeclareMathOperator{\Hom}{Hom}
\DeclareMathOperator{\Mat}{M}

\DeclareMathOperator{\ii}{i}
\DeclareMathOperator{\Ind}{Ind}
\newcommand{\cInd}{{\mathrm{c}\!-\!\Ind}}
\DeclareMathOperator{\inv}{inv}

\newcommand{\fo}{{\mathfrak o}}
\newcommand{\fs}{{\mathfrak s}}
\newcommand{\fp}{{\mathfrak p}}
\newcommand{\cK}{{\mathcal K}}

\newcommand{\cL}{{\mathcal L}}

\newcommand{\cP}{{\mathcal P}}

\newcommand{\fA}{{\mathfrak A}}
\newcommand{\fB}{{\mathfrak B}}
\newcommand{\fP}{{\mathfrak P}}

\newcommand{\bbP}{{\mathbb P}}
\newcommand{\fR}{{\mathfrak R}}

\newcommand{\bG}{{\mathbf G}}

\newcommand{\tG}{{\widetilde{G}}}

\DeclareMathOperator{\rr}{r}
\newcommand{\heis}{{\mathrm{h}}}
\newcommand{\sieg}{{\mathrm{s}}}

\newcommand{\QQ}{{{\mathbb Q}}}

\newcommand{\Qp}{{{\mathbb Q}_p}}
\newcommand{\Zp}{{{\mathbb Z}_p}}
\newcommand{\Gal}{{\mathrm{Gal}}}
\newcommand{\Ad}{{\mathrm{Ad}}}

\newcommand{\fX}{{\mathfrak{X}}}
\newcommand{\nr}{{\mathrm{nr}}}

\usepackage{calrsfs}
\usepackage{anyfontsize}
\usepackage{t1enc}
\hypersetup{colorlinks=true,citecolor=black,linkcolor=black}
\usepackage{comment}

\title[Typical representations for $\Sp_4(F)$]{Typical representations for $\Sp_4(F)$}

\author[Anne-Marie Aubert ]{Anne-Marie Aubert}
\address{Sorbonne Universit\'e and Universit\'e Paris Cit\'e, CNRS, Institut de Math\'ematiques de Jussieu -- Paris Rive Gauche, 
F-75005 Paris, France}
\email{anne-marie.aubert@imj-prg.fr}

\author[Luis Guti\'errez Frez]{Luis Guti\'errez Frez}
\address{Instituto de Ciencias F\'\i sicas y Matem\'aticas, Universidad Austral 
de Chile, Campus Isla Teja SN, Valdivia, Chile}
\email{luis.gutierrez@uach.cl}

\date{\today}

\begin{document}

\maketitle

\begin{abstract} Let $F$ be a non Archimedean local field with odd residual characteristic,  and let $\cK$ a hyperspecial maximal compact subgroup of the $p$-adic symplectic group $G=\Sp_4(F)$. 
Let $\fs$ be an inertial class for $G$ in the Bernstein decomposition of the category of smooth representations of $G$, which is attached to a proper Levi subgroup $L$ of $G$. 

We prove that the $\fs$-typical irreducible representations of $\cK$ are the irreducible components of  $\Ind_{J_\fs}^\cK(\lambda_\fs)$, where  $(J_\fs,\lambda_\fs)$ is an $\fs$-type for $G$ such that $J_\fs\subset \cK$, and
$(J_\fs,\lambda_\fs)$ is a $G$-cover of a Bushnell-Kutzko maximal simple type for $L$.
\end{abstract}

\tableofcontents

\section{Introduction}
Let $F$ be a non Archimedean local field, with ring of integers denoted by $\fo_F$. Let $q$  denote the cardinality of of the residue field $k_F$ of $F$, and $p$ the characteristic of $k_F$. Let $G$ be the  group of $F$-rational points of connected reductive algebraic group defined over $F$.

A \textit{supercuspidal pair} in $G$ is a pair $(L,\sigma)$,  consisting
of a Levi subgroup $L$ of a parabolic subgroup of $G$ and a smooth irreducible supercuspidal representation $\sigma$  of $L$. Given $\pi$ a smooth irreducible representation of $G$, there is a supercuspidal pair $(L,\sigma)$ such that $\pi$ is isomorphic to a  subquotient of a parabolically induced representation from $\sigma$.  The $G$-conjugacy class of $(L,\sigma)$ is uniquely determined and is called the \textit{supercuspidal suppport} of $\pi$.

The group $\fX_\nr(L)$ of unramified characters of $L$ acts by tensorisation on the set of  irreducible supercuspidal representations of $L$, and we denote by $\fs=[L,\sigma]_G$  the $G$-conjugacy class of the pair 
$(L,\fX_\nr(L)\cdot\sigma)$. Let $\fB(G)$ be the set of such $\fs$.  We call  \textit{supercuspidal} every $\fs\in\fB(G)$ such that $L=G$.  By a theorem of Bernstein \cite{Ber}, the  category $\fR(G)$ of smooth representations of $G$ decomposes as a product
\begin{equation}
\fR(G)=\prod_{\fs\in\fB(G)}\fR^\fs(G),
\end{equation} 
where $\fR^\fs(G)$ is the full subcategory consisting of  those representations every irreducible subquotient of which has its supercuspidal suppport in $\fs$. We denote by $\Irr^\fs(G)$ the set of irreducible objects in $\fR^\fs(G)$.

\smallskip

Let $(J, \lambda)$ be a pair consisting of a compact open subgroup $J$ of $G$ and a smooth irreducible representation $\lambda$ of  $J$. 
The pair $(J, \lambda)$ (or simply the representation $\lambda$)  is called \textit{$\fs$-typical}, a notion introduced by Henniart in \cite{He},  if every smooth irreducible representation $\pi$ of $G$ such that $\Hom_J(\lambda,\pi)\ne \{0\}$ is in $\Irr^\fs(G)$.

Following Bushnell-Kutzko's terminology in \cite{BKtyp}, we say that the pair $(J, \lambda)$ (or simply the representation $\lambda$) is an \textit{$\fs$-type} for $G$, if it is $\fs$-typical, and if every representation $\pi$ in $\Irr^\fs(G)$ satisfies $\Hom_J(\lambda,\pi)\ne \{0\}$.  If  $(J, \lambda)$ is an $\fs$-type, then, by \cite[Theorem~4.3]{BKtyp},  the category $\fR^\fs(G)$ is canonically equivalent to the category of (left) modules over the convolution
algebra of compactly supported $\lambda$-spherical functions on $G$.

For $G=\GL_N(F)$ (resp. $G=\SL_2(F)$, with $p$ odd), and $\fs$ supercuspidal, every $\fs$-typical representation of $\GL_N(\fo_F)$ (resp. $\SL_2(\fo_F)$) is actually an $\fs$-type, and is induced from a maximal simple type for $G$, see \cite{He,Pa} (resp. \cite{La1}).

For $G$ an arbitrary $p$-adic group, a representation of $G$ is said to have \textit{depth zero} if it has non-zero vectors that are invariant under the pro-$p$ unipotent radical of a parahoric subgroup of $G$. If $\sigma$ has depth zero, then every representation of $G$ in the orbit $\fX_\nr(L)\cdot\sigma$ has also depth-zero. In this case, we say that $\fs$ has depth-zero. Let $x$ be a vertex in the Bruhat-Tits building of $G$, let $G_{x,0}$ be the parahoric subgroup of $G$ associated to $x$, and let $G_x$ be the maximal compact subgroup of the $G$-normalizer $\widetilde{G}_x$ of $G_{x,0}$. Then $G_x$ is a maximal compact subgroup of $G$ which contains $G_{x,0}$ as a normal subgroup of finite index, and every depth-zero supercuspidal irreducible representation of $G$ is compactly induced from 
the extension to $\widetilde{G}_x$ of the inflation of a cuspidal irreducible representation of the reductive quotient of $G_x$, see \cite{Mor0} or \cite{MP}. If $\fs\in\fB(G)$ is  supercuspidal and has depth zero, then  again every $\fs$-typical representation of $G_x$ is an 
$\fs$-type, see \cite{La2}.

\smallskip

However, for non-supercuspidal $\fs$, in general there exist $\fs$-typical representations that are not $\fs$-types. For $G=\GL_N(F)$, and $\fs\in\fB(G)$ arbitrary, a complete classification of the $\fs$-typical representations of $\GL_N(\fo_F)$, was obtained for $N=2$ in \cite{He}, and for $N=3$, with $q>3$, in \cite{Na2}.  

Let $\cK$ be a maximal compact subgroup of $G$.  If $(J_\fs,\lambda_\fs)$ is an $\fs$-type such that $J_\fs\subset \cK$, then, by Frobenius reciprocity, every irreducible component of the induced representation $\Ind_{J_\fs}^\cK(\lambda_\fs)$ is $\fs$-typical. 
It is expected that all the $\fs$-typical irreducible representation of $\cK$ are obtained in this way. For $\GL_2(F)$ and $\GL_3(F)$, it follows from the classication obtained in \cite{He} and \cite{Na2}. It was also established for $\cK=\GL_N(\fo_F)$
 for all the depth-zero representations of $\GL_n(F)$ in \cite{Na1}, and of  a split classical group $G$, with $q>5$ and $\cK$ a hyperspecial maximal compact subgroup of $G$, as well as for some positive-depth representations in the principal series of $G$ in \cite{MN}.

For non-supercuspidal $\fs$, the method of $G$-covers, introduced and developed in \cite[{\S}8]{BKtyp} (that we will recall in {\S}\ref{sec:cover}), provides a very useful approach to construct $\fs$-types from $\fs_M$-types of Levi subgroups $M$ of $G$
containing $L$.

\medskip

In this article, we focus on the case where $G$ is the symplectic group $\Sp_4(F)$, and prove the following result.

\begin{theorem-intro} \label{thm:typical}  {\rm (See Theorem~\ref{maintheorem})}
Let $\cK$ be a hyperspecial maximal compact subgroup of $G=\Sp_4(F)$, and let $\fs=[L,\sigma]_G\in\fB(G)$, with $L\ne G$. We suppose $p$ odd.

The $\fs$-typical irreducible representations of $\cK$ are the irreducible components of  $\Ind_{J_\fs}^\cK(\lambda_\fs)$, where $(J_\fs,\lambda_\fs)$ is a $G$-cover of a Bushnell-Kutzko maximal simple type for $L$ such that $J_\fs\subset \cK$. In particular, $(J_\fs,\lambda_\fs)$ is an $\fs$-type for $G$.
\end{theorem-intro}

Several of the arguments we use  in the proof of the theorem above should extend to more general groups, and we plan to explore this in future works.

\bigskip

\noindent
\textsc{Notation}.

 For $G$ a group, $g$ and $x$ two elements of $G$, and $X$ a subset of $G$, we write ${}^gx:=gxg^{-1}$ and ${}^gX:=\{{}^gx\,:\, x\in X\}$. 
 Given a representation $\lambda$ of a closed subgroup $J$ of $G$, we denote by ${}^g\lambda$ the representation of ${}^gJ$ defined by ${}^g\lambda(gjg^{-1}):=\lambda(j)$, for any $j\in J$.

\section{Preliminaries on $\Sp_4(F)$}
Let $F$ be a non Archimedean local field, with ring of integers denoted by $\fo_F$, and let $\varpi_F$ be a uniformizer of $F$. We denote by $\fp_F:=\varpi_F\fo_F$ the maximal ideal of $\fo_F$, and by $k_F:=\fo_F/\fp_F$ the residue field of $F$.
From now on,  we assume that the characteristic $p$ of $k_F$ is odd. Let $\bG$  be the algebraic group $\Sp_4$ and let $G=\Sp_4(F)$ be the group of $F$-rational points of $\bG$. 
We realize $G$ as the subgroup of $\tG:=\GL_4(F)$ which consists of elements preserving the non-degenerate alternated form $\langle\;,\;\rangle$ defined by $\langle x,y\rangle:=x H {}^t y$, where 
\begin{equation}H=\left(\begin{smallmatrix} 
0 & 0 &\phantom{-}0 & -1\cr
0 & 0 &-1 &\phantom{-}0\cr
0 & 1 &\phantom{-}0& \phantom{-}0\cr
1 & 0 & \phantom{-}0& \phantom{-}0
\end{smallmatrix}\right).
\end{equation}
In other words, $G=\Sp(V)$, where the $F$-vector space $V=F^4$ is equipped with symplectic form $\langle\;,\;\rangle$.

We denote by $T$ the diagonal maximal torus of $G$:
\[T:=\left\{\left(\begin{smallmatrix} a&0&0&0\cr 
0&c&0&0\cr
0&0&\phantom{-}c^{-1}&0\cr
0&0&0&\phantom{-}a^{-1}\end{smallmatrix}\right)\;:\, a,c\in F^\times\right\},\]
and $B=TU$ be the standard Borel subgroup.

Let $L\supseteq T$ be a Levi subgroup of a parabolic subgroup $P\supseteq B$ of $G$. Then the possibilities for the pair $(L,P)$ are $(T,B)$, $(L_\sieg,P_\sieg)$, $(L_\heis,P_\heis)$, and $(G,G)$,
where $P_\sieg$ is the Siegel parabolic subgroup of $G$:
\begin{align*}
P_\sieg&=\left\{\left(\begin{smallmatrix} a&b\cr 0& \phantom{-}{}^ta^{-1}\end{smallmatrix}\right)\;:\, a\in \GL_2(F),b\in \Mat_2(F)\right\}\cap G,\\
L_\sieg&=\left\{\left(\begin{smallmatrix} a&0\cr 0& \phantom{-}{}^ta^{-1}\end{smallmatrix}\right)\;:\, a\in \GL_2(F)\right\}\simeq\GL_2(F),
\end{align*}
and $P_\heis$ is the Heisenberg parabolic subgroup:
\[P_\heis=\left\{\left(\begin{smallmatrix} a&b&c\cr 0&A&d\cr 0&0&\phantom{-}a^{-1}\end{smallmatrix}\right)\;:\, a\in F^\times,\;
A\in\SL_2(F),b,c,d\in F\right\}\cap G,\]
\[L_\heis=\left\{\left(\begin{smallmatrix} a&0&0\cr 0&A&0\cr 0&0& \phantom{-}a^{-1}\end{smallmatrix}\right)\;:\, a\in F^\times,A\in \SL_2(F)\right\}\simeq F^\times\times\SL_2(F).\]

\subsection{Lattices and parahoric subgroups} \label{subsec:lat}

For $\cL$ an $\fo_F$-lattice in $V$, we denote by $\cL^\natural$ the dual lattice
\[\cL^\natural:=\left\{v\in \cL\,:\,\langle v,L\rangle\subset \fp_\cL\right\}.\]
The lattice $\cL$ is called almost self-dual if $\cL\supset \cL^\natural \supset\fp_F \cL$. 
If $\cL$ is almost self-dual, its stabilizer $K_\cL$ in $G$ is a maximal compact subgroup of $G$, and we will
denote by $K_\cL^1$ the pro-$p$ unipotent radical of $K_\cL$, that is the subgroup consisting of those element $k\in K_\cL$ which induce the identity map on the $k_F$-vector spaces $\cL/\cL^\natural$ and $\cL^\natural/\fp_F \cL$.
The form $\langle\;,\;\rangle$ induces non-degenerate alternated forms on $\cL/\cL^\natural$ and $\cL^\natural/\fp_F \cL$ by setting
\begin{align*}
\langle v_1+\cL^\natural,v_2+\cL^\natural\rangle&:=\langle v_1,v_2\rangle+\fp_F,\quad\quad \quad\: \:\text{ for } v_1,v_2\in \cL\\
\langle w_1+\fp_F \cL,w_2+\fp_F \cL\rangle&:=\varpi_F^{-1}\langle w_1,w_2\rangle+\fp_F,\quad\text{ for } w_1,w_2\in \cL^\natural
\end{align*}
We observe that the quotient $\bar K_\cL:=K_\cL/K^1_\cL$ is isomorphic to the finite group 
$\Sp(\cL/\cL^\natural)\times\Sp(\cL^\natural/\fp_F \cL)$, and
$K_\cL$ is a (maximal) parahoric subgroup of $G$, it is equal to its normalizer in $G$.

Let $(e_i)_{1\le i\le 4}$ be the standard basis of $V$, that is, $e_1=(1,0,0,0)$, $e_2=(0,1,0,0)$,
$e_3=(0,0,1,0)$ and $e_4=(0,0,0,1)$. We will take for $\cL$  the following almost self-dual $\fo_F$-lattice in $V$:
\[\cL:=\fo_F e_1\oplus  \fo_F e_2 \oplus \fo_F e_3 \oplus \fo_F e_4,\]
and set  $K:=K_{\cL}=\Sp_4(\fo_F)$.

Let  $u:=\left(\begin{smallmatrix} 0& \Id_2\cr
\varpi_F \Id_2& 0\end{smallmatrix}\right)\in\GL_4(F)$. The element $u$ does not belong to $\Sp_4(k_F)$. We set  $K':=u^{-1}K u$. We have
\begin{equation}
K'=\left(\begin{smallmatrix} 
\fo_F & \fo_F &\fp^{-1}_F & \fp^{-1}_F\cr
\fo_F & \fo_F &\fp^{-1}_F & \fp^{-1}_F\cr
\fp_F & \fp_F &\fo_F & \fo_F\cr
\fp_F & \fp_F &\fo_F& \fo_F
\end{smallmatrix}\right)\cap G.
\end{equation}
The groups $K$ and $K'$ are hyperspecial maximal compact subgroups of $G$. Their reductive quotients  are both isomorphic to $\Sp_4(k_F)$. 

\smallskip

Let $B$  be the standard Borel subgroup, and $I$ the standard Iwahori subgroup of $G$. 
We recall the Iwasawa decompositions \cite[(4.4.1), (4.4.6)]{BT1}:
\begin{equation} \label{eqn:Iwasawa dec}
G=BK=BK'.
\end{equation}
We set 
\begin{equation}
s_0=\left(\begin{smallmatrix} \phantom{-}0&\phantom{-}0&0&\phantom{-}\varpi_F^{-1}\\\phantom{-}0&\phantom{-}0&1&0\\\phantom{-}0&-1&0&0\\-\varpi_F&\phantom{-}0&0&0\end{smallmatrix}\right),\;s_1:=\left(\begin{smallmatrix} 0&1&0&0\cr 
1&0&0&0\cr
0&0&0&1\cr
0&0&1&0\end{smallmatrix}\right)\quad\text{ and }
\quad
s_2=\left(\begin{smallmatrix}1&\phantom{-}0&0&0\\0&\phantom{-}0&1&0\\0&-1&0&0\\0&\phantom{-}0&0&1\end{smallmatrix}\right).
\end{equation}
The standard parahoric subgroups of $G=\Sp_4(F)$ correspond bijectively to proper subsets of $\{s_0,s_1,s_2\}$. If $S'\subset\{s_0,s_1,s_2\}$, then the corresponding standard parahoric subgroup $\cP_{S'}$ is the group $\langle I, S'\rangle$. If $S'\subset\{s_1,s_2\}$, then $\cP_{S'}$ is the inverse image of $\bbP_{S'}(k_F)$ in $\Sp_4(\fo_F)$. The other standard parahoric subgroups of $G$, that is, the ones of the form $\cP_{S'}$ where $s_0\in S'$, do not lie inside $\Sp_4(\fo_F)$.
Hence, in addition to the maximal parahoric subgroups and the Iwahori subgroup $I$, we have two intermediate parahoric subgroups contained in $\Sp_4(\fo_F)$: the Siegel parahoric subgroup
\begin{equation} \label{eqn:Sparahoric}
\cP_\sieg:=\left\{\left(\begin{smallmatrix} a&b\cr \fp_F& \phantom{-}{}^ta^{-1}\end{smallmatrix}\right)\;:\, a\in \GL_2(\fo_F),b\in \Mat_2(\fo_F)\right\}\cap G, 
\end{equation}
which is the inverse image of the Siegel parabolic subgroup of $\Sp_4(k_F)$, and the Heisenberg parahoric subgroup
\begin{equation} \label{eqn:Hparahoric}
\cP_\heis:=\left\{\left(\begin{smallmatrix}  a&b&c\cr \fp_F&A&d\cr \fp_F&\fp_F&a^{-1}\end{smallmatrix}\right)\;:\, a\in \fo_F^\times,\;
A\in\SL_2(\fo_F),b,c,d\in \fo_F\right\}\cap G,
\end{equation}
which is the inverse image of the Heisenberg parabolic subgroup of $\Sp_4(k_F)$.

\subsection{Parahoric restriction for $\Sp_4(F)$} \label{subsec: par_res}
Let $\cP$ be a parahoric subgroup of $G$, and let $\cP_+$ and $\overline\cP$ denote the pro-$p$-unipotent radical and the reductive quotient of $\cP$, respectively. The sequence
\begin{equation} \label{eqn:parahoric}
1\to \cP_+\to \cP\to \overline\cP\to 1
\end{equation}
is exact, and $\overline\cP$ is the group of $k_F$-rational points of a connected reductive algebraic group defined over $k_F$. 

We denote by $\fR(\overline\cP)$ the category of representations of $\overline\cP$, and by
\begin{equation}
\infl_{\overline\cP}\colon\fR(\overline\cP)\to \fR(\cP)
\end{equation} 
the inflation functor along the projection $\cP\to\overline\cP$.  Its right adjoint  is the functor 
$\inv^{\cP_+}\colon\cP\to\overline\cP$ that sends  a representation $(\tau,V)$  to $(\tau,V^{\cP_+})$, where $V^{\cP_+}$ is the space of $\cP_+$-fixed vectors in $V$. 
Then $\inv^{\cP_+}$ defines an exact functor from $\fR(\cP)$ to $\fR(\overline\cP)$.

The {\em parahoric induction} functor for $\cP$ is defined to be
\begin{equation} \label{eqn:parahoric induction}
\ii^G_{\overline\cP}:=\cind_\cP^G\circ \infl_{\overline\cP}\colon\fR(\overline\cP)\to\fR(G).
\end{equation}
Its right adjoint is the {\em parahoric restriction} functor for $\cP$:
\begin{equation} \label{eqn:res_par}
\rr^G_{\overline\cP}:=\Res^G_\cP\circ \inv^{\cP_+}\colon\fR(G)\to\fR(\overline\cP).
\end{equation}

Let  $\cK\in\{K,K' \}$ and let $P$ be a parabolic subgroup of $G$,  with Levi subgroup $L$.  We set $\cK_L:=\cK\cap L$. Then $\cK_L$ is a parahoric subgroup of $L$. We write $\cK_P:=\cK\cap P$ and $\cK_U:=\cK\cap U$, where $U$ is the unipotent radical subgroup  of $P$. 
Thanks to $\cK_L\cap \cK_U=\{1\}$ we have that $\cK_P=\cK_L\cK_U$. Thereby we have an exact sequence
\begin{equation} \label{eqn:parahoric2}
1\to \cK_U\to \cK_P\to \cK_L\to 1.
\end{equation}
We denote by $\infl_{\cK_L}^{\cK_P}$ the natural functor from $\fR(\cK_L)$ to $\fR(\cK_P)$ along the projection described above.

\begin{proposition} \label{Siegel_K_L}  
The  following diagram of functors
\begin{equation}
\xymatrix{\fR(L) \ar^{\Res^L_{\cK_L}}[r] \ar_{\ii_{L,P}^G}[d] & \fR(\cK_L) \ar^{\Ind_{\cK_P}^{\cK}\circ\infl_{\cK_L}^{\cK_P}}[d] \\
\fR(G) \ar^{\Res^G_{\cK}}[r]  & \fR(\cK)}
\end{equation}
is commutative.
\end{proposition}
\begin{proof}
Let $(\sigma, V_{\sigma})$ be an irreducible  smooth representation of $L$. We consider first  
$\Ind_{\cK_P}^{\cK}\circ\infl_{\cK_L}\circ \Res^L_{\cK_L}$. By definition, we have
\begin{align*}
&\left(\Ind_{\cK_P}^{\cK}\circ\infl_{\cK_L}^{\cK_P}\circ \Res^L_{\cK_L}\right)(\sigma)\\
&=\{f\colon\cK\to V_{\sigma}:f(pk)=\left(\infl_{\cK_L}^{\cK_P}\circ \Res^L_{\cK_L}\right)(\sigma)(p)f(k), \;p=\ell u\in \cK_{P},\;  \;k\in \cK  \}\\
&=\{f\colon\cK\to V_{\sigma}:f(\ell uk)=\sigma(\ell)f(k), \;\ell \in \cK_{L},\;u\in U,\;  \;k\in \cK  \}
\end{align*}
endowed with the right action by $\cK$. \\
We consider now  $\Res^G_{\cK}\circ \ii_{L,P}^G$.  The  induced representation $\ii_{L,P}^G(\sigma) $ is given by 
\begin{eqnarray*}
\ii_{L,P}^G(\sigma) &=&\{f\colon G\to V_{\sigma}:f(pg)=\sigma(p)f(g), \;p\in P,  \;g\in G  \}
\end{eqnarray*}
endowed with right action by $G$. Since $G=P\cK$ by \eqref{eqn:Iwasawa dec}, we see that each $f\in V$ is completely determined by its restriction to $\cK$. Thus we can conclude that $(\Res^G_{\cK}\circ \ii_{L,P}^G)(\sigma)$ is simply 
\begin{eqnarray*}
\left\{f\colon \cK\to V_{\sigma}\,:\, f(\ell uk)=\sigma(\ell)f(k), \;\ell\in \cK_{L},\;u\in U,\; \; k\in \cK\right  \}
\end{eqnarray*}
 with the right action by $\cK$.  This implies 
 \[
\Ind_{\cK_P}^{\cK}\circ\infl^{\cK_P}_{\cK_L}\circ \Res^L_{\cK_L} =\Res^G_{\cK}\circ \ii_{L,P}^G,
 \]
 and the  result follows.
\end{proof}

We set 
\begin{equation}
\cK_{P,+}:=\cK_+\cap P\quad\text{and}\quad  \overline{\cK}_P:=\cK_P/ \cK_{P,+}.
\end{equation}
The group $\overline\cK_P$ is a parabolic subgroup of $\overline\cK$ with Levi factor $\overline{\cK}_L$, and we denote by 
\begin{equation}  
\ii_{\overline{\cK}_L}^{\overline\cK}\colon \fR(\overline{\cK}_L)\to\fR(\overline\cK)
\end{equation}
the corresponding parabolic induction functor. 
We notice that
\begin{equation} \label{eqn:red-qt}
\overline{\cK}_L=\begin{cases}
\overline I\simeq k_F^{\times}\times k_F^{\times}, &\text{if $L=T$,}\, \; \\ \overline \cP_\sieg\simeq\GL_2(k_F), &\text{if $L=L_\sieg$,}\, \; \\ \overline \cP_\heis\simeq k_F^{\times}\times \SL_2(k_F), &\text{if $L=L_\heis$.}\end{cases}
\end{equation}

\begin{proposition} \label{Inv_K_L}
The following diagram of functors
\begin{equation} \label{eqn:Inv_K_L}
\xymatrix{
\fR(\cK_L) \ar^{\inv_{  \cK_{L,+}}}[r] \ar_{\Ind_{\cK_P}^{\cK}\circ\infl_{\cK_L}^{\cK_P}}[d] & \fR(\overline{\cK}_L) \ar^{\ii_{\overline{\cK}_L}^{\overline{\cK}}}[d] \\
\fR(\cK) \ar^{\inv_{\cK^+}}[r]  & \fR(\overline{\cK})}
\end{equation}
commutes.
\end{proposition}
\begin{proof}
Let $(\tau,V_{\tau})\in \fR(\cK_L)$, and we compute firstly the composition
 $\ii_{\overline{\cK}_L}^{\overline\cK}\circ \inv_{  \cK_{L,+}}$. We observe that $\inv_{  \cK_{L,+}}(\tau)(\bar k)=\tau(k)$, for 
$\bar k\in \cK/  \cK_{L,+}$, by definition.  Let $\overline\cK_P=\overline\cK_L\overline\cK_U$  denote
the Levi decomposition of $\overline\cK_P$. By applying the functor $\ii_{\overline{\cK}_L}^{\overline\cK}$, we obtain
\begin{align*}
&\ii_{\overline{\cK}_L}^{\overline{\cK}}(\inv_{\cK_{L,+}}(\tau))\\
&=\left\{\bar f\colon\overline{\cK}\to V_{\tau}\colon\bar f(\bar p \bar x)=\infl_{\overline{\cK}_L}^{\overline{\cK}_P}
(\inv_{\overline{\cK}_L^+}(\tau))(\bar p)f( \bar x), \;\bar p\in\overline\cK_P, \;\bar x\in\overline\cK \right\}
\\
&=\left\{ \bar f\colon\overline\cK\to V_{\tau}\,:\,\bar f(\bar\ell   \bar u \bar x)=\inv_{  \cK_{L,+}}(\tau)
(\bar\ell)f( \bar x), \; \overline\ell  \in\overline\cK_L, \;\bar u\in\overline{\cK}_U, \;\bar x\in\overline{\cK} \right\}
\\
&=\left\{ \bar f\colon\overline\cK\to V_{\tau}:\bar f(\bar p \bar x)=\tau(\ell)f( \bar x), \;\bar\ell\in\overline{\cK}_L,\bar u\in  \overline\cK_U, \;
\bar x\in\overline{\cK} \right\}.
\end{align*}
On the other hand, we see that $\Ind_{\cK_P}^{\cK}\left(\infl_{\cK_L}^{\cK_P}(\tau)\right)$ is 
\[
\left\{f\colon\cK\to V_{\tau}:f(px)=\infl_{\cK_L}^{\cK_P}(\tau)(p)f(x),\:p\in\cK_P,\;x\in \cK \right\}.
\]
Thus the subspace of $\cK^+$-fixed vectors $\inv_{\cK^+}\left(\Ind_{\cK_P}^{\cK}\left(\infl_{\cK_L}^{\cK_P}(\tau)\right)\right)$ is given by
\[
\left\{f\colon\cK\to V_{\tau}\colon f(\ell uxk^+)=\tau(\ell)f (x),\:\ell u\in\cK_P,\;x\in \cK,\;k^+\in \cK^+ \right\}.
\]
Then the map  from $\inv_{\cK^+}\left(\Ind_{\cK_P}^{\cK}\left(\infl_{\cK_L}^{\cK_P}(\tau)\right)\right)$ to 
$\ii_{\overline{\cK}_L}^{\overline\cK}(\inv_{\cK^+}(\tau))$ given by $f\mapsto \bar f$ defined by $\bar f(\bar x):=f(x)$
is an isomorphic of representations. The result then follows from this.
\end{proof}

\begin{corollary} \label{prop_res_Siegel}
The following diagrams of functors
\[
\xymatrix{
\fR(L) \ar^{\rr^L_{\overline{\cK}_L}}[r] \ar_{\ii_{L,P}^G}[d] & \fR(\overline{\cK}_L) \ar^{\ii_{\overline{\cK}_L}^{\overline\cK}}[d] \\
\fR(G) \ar^{\rr^G_{\overline\cK}}[r]  & \fR(\overline{\cK})}\quad\quad\text{and}\quad\quad
\xymatrix{
\fR(\overline{\cK}) \ar^{\ii^G_{\overline\cK}}[r] \ar_{\rr^{\overline\cK}_{\overline{\cK}_L}}[d] & \fR(G) \ar^{\rr_{L,P}^G}[d] \\
\fR(\overline{\cK}_L) \ar^{\ii^L_{\overline{\cK}_L}}[r]  & \fR(L)}
\]
are commutative.
\end{corollary}
\begin{proof}
The commutativity of the first diagram  follows from the combination of Proposition \ref{Siegel_K_L} and  Proposition \ref{Inv_K_L} while  the second one is deduced by adjunction.
\end{proof}

\section{$G$-covers} \label{sec:cover}
Let $P=LU$ be a parabolic subgroup of $G$ with Levi factor $L$, and let $P^-=L  U^-$ be the opposite parabolic subgroup. A compact open subgroup $J$ of $G$ is said to \textit{decompose with respect} to $(U,L,  U^-)$ if 
\begin{equation} \label{eqn:Idec}
J=(J\cap U)\cdot (J\cap L)\cdot (J\cap   U^-).
\end{equation}
If $V$ is a smooth representation of $J$, we denote by $V^\lambda$ the $\lambda$-isotypic part of $V$, {\it i.e.}, the sum of all $J$-invariant subspaces of $V$ that are isomorphic to $\lambda$.

Let $J$ (resp. $J_L$) be a compact open subgroup of $G$ (resp. $L$), and $\lambda$ (resp. $\lambda_L$) an irreducible smooth representation of $J$ (resp. $J_L$). We suppose that the pair $(J,\lambda)$ is a \textit{$G$-cover} of the pair $(J_L,\lambda_L)$ (see \cite{BKtyp,Bl1}), {\it i.e.,} that,  for any opposite pair of parabolic subgroups $P=LU$ and $P^-=L  U^-$ with  with Levi factor $L$, the following conditions are satisfied:
\begin{enumerate}
\item $J$ decomposes with respect to $(U,L,  U^-)$;
\item $\lambda|_{J_L}=\lambda_L$ and $J\cap U, J\cap  U^-\subset\ker(\lambda)$;
\item for any smooth representation $V$ of $G$, the natural map from $V$ to its Jacquet module $V_U$ induces an injection on $V^{\lambda}$.
\end{enumerate}

\bigskip

\noindent
\textsc{Depth-zero case}. 
Recall that,  for $L$  a Levi subgroup of $G$, we are setting $\cK_L:=\cK\cap L$, where $\cK$ is  a maximal compact subgroup of $G$.
Let $\fs_L=[L,\sigma]_L$, where $\sigma$ has depth-zero.   By \cite{Mor0}, we have $\sigma=\cInd_{\cK_L}^L\tau_L$, where $\tau_L$ is the inflation to $\cK_L$ of an irreducible cuspidal representation $\overline\tau_L$ of $\overline{\cK}_L$. 
Let  $\cP$  be parahoric subgroup  of $G$ defined in \eqref{eqn:Sparahoric} such that  $\cK_L\subset \cP$. Then we set 
\begin{equation} \label{eqn:depth-zero-type}
\lambda:=\infl_{\overline\cK_L}^{\cP}\overline\tau_L.
\end{equation}
By \cite{Mor}, the pair $(\cP,\lambda)$ is a $G$-cover of the pair $(\cK_L,\tau_L)$.
\begin{proposition} \label{lem:Inf_P_s} We keep the notation above. Then
\begin{equation}\label{Inf_P_s}\Ind_{\cP}^\cK\lambda=\Ind_{\cK_P}^\cK\circ\infl_{\cK_L}^{\cK_P}\tau_L.\end{equation}
\end{proposition}
\begin{proof}We shall start by describing the right side:\begin{align*}\Ind_{\cK_P}^{\cK}\infl_{\cK_L}^{\cK_P}\tau_L&=\Ind_{\cK_P}^\cK\infl_{\cK_L}^{\cK_P}\infl_{\overline\cK_L}^{\cK_L}\overline\tau_L\\&:=\{ f:\cK\to V_{\overline\cK_L}:f(px)=\infl_{\cK_L}^{\cK_P}\tau_L(p)f(x), \;p\in \cK_P,\;x\in  \cK\}\\&:=\{ f:\cK\to V_{\overline\cK_L}:f(px)=\tau_L(x_L)f(x), \;p=x_Lu\in \cK_P,\;x\in  \cK\}\\&:=\{ f\colon\cK\to V_{\overline\cK_L}:f(px)=\overline\tau_L(\bar x_L)f(x), \;p=x_Lu\in \cK_P,\;x\in  \cK\},\end{align*}where $\bar x_L$ is the image of $x_L$ under the natural projection  from  $\cK_L$ to $\overline{\cK_L}$. But $\overline\tau_L(\bar x_L)$ is precisely the definition of $\infl_{\overline{\cK}_L}^{\cP}\overline\tau_L$ by using the natural projection  from $\cP$   to $\overline{\cK_L}$, which implies the equality \eqref{Inf_P_s}. \end{proof}

\section{Typical representations of $G=\Sp_4(F)$}
\subsection{Preparatory results}
Let  $\fs=[L,\sigma]_G\in\fB(G)$ with $L\ne G$. We set $\fs_L:=[L,\sigma]_L$ and $\cK\in\{K,K' \}$. We recall that $\cK_L=\cK\cap L$.
An irreducible representation of $\cK_L$ which is not $\fs_L$-typical will be called \textit{$\fs_L$-atypical}.

\begin{proposition}\label{restrictionK_L}
The restriction $\Res^L_{\cK_L}(\sigma)$ admits a decomposition
\begin{equation} \label{eqn:restrictionK_L}
\Res^L_{\cK_L}(\sigma)=\tau_L\oplus\tau_L',\end{equation}
such that $\tau_L$ is  an $\fs_L$-type, and every irreducible $\cK_L$-subrepresentation of $\tau_L'$ is $\fs_L$-atypical.
\end{proposition}
\begin{proof}
We will do a case by case analysis according to the possibilities for $L$.\\
$\bullet$ For  $L=T\cong F^\times\times F^\times$, we have  $\sigma=\mu_1\otimes \mu_2$, where   $\mu_1$, $\mu_2$ are linear characters of $F^{\times}$. If  $\mu_1$, $\mu_2$ are unramified (that is, $\mu\vert_{\fo_F^\times}=1$) we get that 
$\Res^L_{\cK_L}(\sigma)$ is the trivial character of $\cK_L$ and thus $\tau_L=1$  and  $\tau_L'=0$.
Suppose next that $\mu_i$ has level $m_i>0$ (that is, $\mu_i\vert_{1+\fp_F^{m_i+1}}=1$ and $\mu\vert_{1+\fp_F^{m_i}}\neq 1$, \cite[{\S}1.8]{BH}), for $i=1,2$. By Clifford theory  $\mu_i\vert_{\fo_F^{\times}}$  factors through  a multiple of a character $\vartheta_i$ of $\fo_F^{\times}/(1+\fp_F^{m_i+1})$, namely 
$\mu_i\vert_{\fo_F^{\times}}=n_i\vartheta_i$.
Thus 
\[
(\mu_1\otimes \mu_2)\vert_{\cK_L} \text{ factors through  }n_1\vartheta_1\otimes n_2\vartheta_2. 
\]
This last fact  implies  $\tau_L=\mu_1\vert_{\fo_F^{\times}}\otimes\mu_2\vert_{\fo_F^{\times}}$  and  $\tau_L'=0$ satisfy the proposition.
Finally, when  $\mu_1$ is unramified and $\mu_2$ has level $m_2>0$, we obtain that
$(\mu_1\otimes  \mu_2)\vert_{\cK_L}$ factor through $1\otimes  n_2\vartheta_2$. The case where $\mu_1$ is  primitive of level $m_1$  and $\mu_2$  is unramified is analogous.\\
$\bullet$ For $L=L_\sieg\simeq\GL_2(F)$, we have $\cK_L\cong \GL_2(\fo_F)$. By \cite[(8.4.1)]{BK1}, there exists a (uniquely determined up to $\GL_2(F)$-conjugacy) simple type $(J_0,\lambda_0)$ 
such $\lambda_0$ is contained in the restriction of $\sigma$ to $J_0$. It is an $\fs_L$-type, and it is given by a simple stratum $[\fA_0,n,0,\beta_0]$. In particular, $\fA_0$ is  a hereditary $\fo_F$-order,  $n$ a positive integer, 
and $\beta_0$  an element of $\Mat_2(F)$  with $\fA_0$-valuation at least $-n$ (see \cite[(1.5) \& (2.3.2)]{BK1}). For a positive integer $m$, we set $U_{\fA_0}^m:=1+\fP_0^m$, where $\fP_0$ is the Jacobson radical of $\fA_0$. 
Let $E=F[\beta_0]$ and let $\Lambda_0$ be a representation  of $E^{\times}J_0$  such that $\Lambda_0\vert_{J_0}=\lambda_0$. We have 
\begin{equation} \label{eqnJ0}
J_0:= U_{\fA_0}^{\lfloor(n+1)/2\rfloor}\quad\text{and}\quad\sigma\cong \cind_{E^{\times}J_0}^{\GL_2(F)}\Lambda_0.
\end{equation} 
We define
\begin{equation} \label{eqn:rho1-tauL}
\rho_0:=\ind^{U_{\fA_0} }_{J_0}\lambda_0\quad\text{and}\quad \tau_L:=\ind_{U_{\fA_0}}^{\cK_L} \rho_0.
\end{equation}
By \cite[A.{\S}3]{He}, the representation $\tau_L$ is the unique $\fs_L$-typical representation of $\cK_L$, and is an $\fs_L$-type. Then the result  follows by setting
\[
\tau_L':=\bigoplus_{g\in \cK_L\backslash\GL_2(F)/E^{\times} U_{\fA_0}\atop g\neq 1}\ind_{\cK_L\cap U_{\fA_0}^{g}}^{\cK_L}(\rho_0^g\vert_{\cK_L\cap   U_{\fA_0}^{g}}).
\]
\\
$\bullet$  For  $L=L_\heis\cong\SL_2(F)\times F^\times$, the representation $\sigma$ is of  the form $\sigma=\mu\otimes \sigma_0$, 
where $\mu$ is a linear character of $F^\times$ and $\sigma_0'$ a supercuspidal irreducible representation of $\SL_2(F)$. By \cite{BK3}, we know that 
$\sigma_0'=\cind ^{\SL_2(F)}_{J_0'}\lambda'_0$, where $(J_0',\lambda'_0)$ is a maximal simple type for  $\SL_2(F)$ with $J_0'\subseteq \SL_2(\fo_F)$. 
We wish to  describe 
 \[\Res^L_{\cK_L}(\mu\otimes  \sigma_0')=
 \mu\vert_{\fo_F^{\times}}\otimes \Res^{\SL_2(F)}_{\SL_2(\fo_F)}\sigma_0'.
\]
Regarding the $\SL_2(\fo_F)$-restriction above,  we observe, by using Mackey formula, that 
 \[
\Res^{\SL_2(F)}_{\SL_2(\fo_F)}\sigma_0=\ind_{J_0'}^{\SL_2(\fo_F)}\lambda_0'\,\oplus\,  \bigoplus_{g\in \SL_2(\fo_F)\backslash\SL_2(F)/J_0'\atop g\neq 1} \ind_{\SL_2(\fo_F)\cap (J_0')^{g}}^{\SL_2(\fo_F)}(\lambda_0')^{g},
\]
By \cite[Theorem 2.5(ii)]{La1}, the representation $\ind_{J_0'}^{\SL_2(\fo_F)}\lambda'_0$ is the unique $\fs_0$-typical representation of $\SL_2(F)$.  Thereby, this case is satisfied by setting 
\begin{align*}
\tau_L&:=\mu\vert_{\fo_F^{\times}}\otimes \ind_{J_0'}^{\SL_2(\fo_F)}\lambda_0',\\ \tau_L'&:=\mu\vert_{\fo_F^{\times}}\otimes (\bigoplus_{g\in \SL_2(\fo_F)\backslash\SL_2(F)/J_0'\atop g\neq 1} 
\ind_{\SL_2(\fo_F)\cap (J_0')^{g}}^{\SL_2(\fo_F)}(\lambda_0')^{g}).
\end{align*}
\end{proof}

\begin{lemma} \label{lem:decomp} The $\fs$-typical irreducible representations of $\cK$ are the subrepresentations of $(\ind^{\cK}_{\cK_P}\circ \infl^{\cK_P}_{\cK_L})(\tau_L)$, where $\tau_L$ is an $\fs_L$-type.
\end{lemma}
\begin{proof} 
Let  $\tau$ be an $\fs$-typical representation of $\cK$. Then $\tau$ occurs in $\Res^G_{\cK}(\pi)$ for some $\pi\in \Irr^\fs(G)$.  Thus, it
appears in $(\Res^G_{\cK}\circ\ii_{L,P}^G)(\sigma)$, where $\sigma$ is an irreducible supercuspidal representation of $L$ such that $\fs=[L,\sigma]_G$. 

By Proposition \ref{restrictionK_L} we have
\begin{equation} \label{eqn:dec}
\Res^L_{\cK_L}\sigma=\tau_L\oplus \tau'_L,
\end{equation}
where $\tau_L$ is an $\fs_L$-type, and every irreducible $\cK_L$-subrepresentation of $\tau_L'$ is $\fs_L$-atypical.  
By applying Proposition~\ref{Siegel_K_L} to $(L,\sigma)$,  we  get
\begin{equation} \label{eqn:commutative}
(\Res^G_{\cK}\circ\ii_{L,P}^G)(\sigma)\cong(\ind^{\cK}_{ \cK_P}\circ\infl_{ \cK_L}^{ \cK_P}\circ \Res^L_{\cK_L})(\sigma),
\end{equation}
and hence $\tau$ is a subrepresentation of 
\begin{equation} \label{eqn:commutative2}
(\ind^{\cK}_{ \cK_P}\circ\infl_{ \cK_L}^{ \cK_P}\circ \Res^L_{\cK_L})(\sigma)
=\ind^{\cK}_{\cK_P}\circ \infl^{\cK_P}_{\cK_L}\tau_L
\oplus \ind^{\cK}_{\cK_P}\circ \infl^{\cK_P}_{\cK_L}\tau'_L.
\end{equation}
Let $\tau'_0$ be an irreducible subrepresentation of $\ind^{\cK}_{\cK_P}\circ \infl^{\cK_P}_{\cK_L}\tau'_L$.  Thus, there exists  an irreducible subrepresentation 
$\tau'_{L,0}$ of $\tau_L'$ such that $\tau'_0$ appears in $(\ind^{\cK}_{\cK_P}\circ \infl^{\cK_P}_{\cK_L})(\tau'_{L,0})$. By \eqref{eqn:commutative}, the representation $\tau'_0$ occurs in $\Res^G_{\cK}(\ii_{L,P}^G\sigma)$.

Since $\tau_{L,0}'$ is $\fs_L$-atypical,  there exists an irreducible smooth representation $\sigma'$ of $L$ which contains $\tau_{L,0}'$, such that $\fs'_L:=[L,\sigma']_L\ne \fs_L$. The induced representation $\ii_{L,P}^G\sigma'$  has inertial support $\fs'_L$. 
By applying Proposition~\ref{Siegel_K_L} to $(L,\sigma')$, we obtain that  $\tau'_0$ occurs in  $\Res^G_{\cK}(\ii_{L,P}^G\sigma')$.  Therefore, the representation $\tau'_0$ is $\fs$-atypical. As a consequence,
since $\tau$ is $\fs$-typical, it follows from \eqref{eqn:commutative2}, that it must necessarily appear as subrepresentation of $\ind^{\cK}_{\cK_P}\circ \infl^{\cK_P}_{\cK_L}\tau_L$.
\end{proof}

\subsection{Main Result}
We keep the notation above: $\fs=[L,\sigma]_G$,  $\cK\in \{K,K'\}$,  $\cK_P=\cK\cap P$,  $\cK_L=\cK\cap L$ and $\cK_U:=\cK\cap U$. Thus we see that  $\cK_L\cong \cK_P/\cK_U$. \\
Let $(J_L,\lambda_L)$ be an $\fs_L$-type for $L$.  We now recall the construction of a $G$-cover of $(J_L,\lambda_L)$, following  \cite{MS} and \cite[\S 3]{BHS} (see also \cite{BB}, and, in the case of the Siegel parabolic subgroup, \cite{GKS}). To this end, we first observe  that $L$ is the stabilizer of the self-dual decomposition for $V=F^4$ 
\[
V=F^2\oplus F^2,\quad V=F\oplus F^2\oplus F^2\quad\text{ or } \quad V=F\oplus F\oplus F\oplus F
\]
according to  $L=L_\sieg$ , $L=L_\heis$ or $L=T$, respectively.
Then there exist a skew semisimple stratum in $[\Lambda,0,-,\beta]$ and a skew semisimple character $\theta$ of $H^1_\Lambda:=H^1(\Lambda,\beta)$ such that the conditions of \cite[\S 3.5]{BHS} are satisfied. We set $J:=J(\beta,\Lambda)$, and define 
\begin{equation}
J_\fs:=(J\cap P)\cdot H^1_\Lambda\quad\text{and}\quad J_\fs^1:=(J^1\cap P)\cdot H^1_\Lambda.
\end{equation}
Up to conjugating $J_\fs$ in $G$, we may suppose that $J_\fs$ is contained in $\cK$. 
Let $\kappa_P$  be the natural representation of $J_\fs=H^1(J\cap P)$ on the $(J\cap U)$-fixed vectors in $\kappa$. Let $\theta_P$ be the character of $H_P^1=H^1(J^1\cap U)$ which extends $\theta$ and is trivial on  $J^1\cap U$. By \cite[Theorem 5.3]{MS}, the pair $(J_\fs,\lambda_\fs)$ is a $G$-cover of $(J_L,\lambda_{L})$.  We write 
 \[J_\fs^+:=J_\fs\cap U\quad\text{and}\quad J_\fs^-:=J_\fs\cap U^-.\]
By \eqref{eqn:Idec}, we have 
\begin{equation} \label{eqn:intersect}
\cK_P\cap J_\fs=J_LJ_\fs^+\quad\text{and}\quad\cK_P^+\cap J_\fs=J_\fs^+,
\end{equation}  
where $\cK_P:=\cK\cap P$ and $\cK_P^+:=\cK_U=\cK\cap U$. Since $J_\fs^+ \subset \cK_P^+$, we obtain
 \begin{equation}
  J_L\cdot J_\fs^+\subset J_L\cdot \cK_P^+\subset \cK_P\quad\text{and}\quad \cK_P\cap J_\fs^-=\{1\}.
 \end{equation}
We observe that $J_L\cdot J_\fs^+$ is a group (since $J_\fs^+$ normalizes $J_L$).
We set
 \begin{equation}
 \tau_L:=\Ind_{J_L}^{\cK_L}\lambda_L.
\end{equation}
 
Now we will prove our main result:
\begin{theorem} \label{maintheorem} Then $\fs$-typical irreducible representations of $\cK$ are  the irreducible components of  $\Ind_{J_\fs}^\cK(\lambda_\fs)$.
\end{theorem}
\smallskip
\begin{remark} {\rm
In the case where $\sigma$ has depth zero, we can prove it as follows:
Let $\tau$ be a $\fs$-typical representation of $\cK$. Let $\lambda$ as in \eqref{eqn:depth-zero-type}.
Since $(\cP,\lambda)$ is a $G$-cover of the pair $(\cK_L,\tau_L)$, and hence a $\fs$-type, the pair $(J_\fs,\lambda_\fs):=(\cP,\lambda)$ satisfies  the theorem by combining  Lemma~\ref{lem:decomp} and Proposition~\ref{lem:Inf_P_s}.}
\end{remark}
\smallskip
\begin{proof} 
By Lemma~\ref{lem:decomp}, we are reduced to prove that 
 \begin{equation} \label{eqn:study}
 \Hom_{\cK}(\ind^{\cK}_{\cK_P}\circ \infl^{\cK_P}_{\cK_L}\tau_L,\ind_{J_\fs}^{\cK}\lambda_\fs)\ne 0.
 \end{equation}
By Frobenius reciprocity,  \eqref{eqn:study} equals
 \begin{equation} 
 \Hom_{J_\fs}(\Res_{J_\fs}^{\cK}\circ \ind^{\cK}_{\cK_P}\circ \infl^{\cK_P}_{\cK_L}\tau_L,\lambda_\fs).
 \end{equation}
 By the Mackey formula, it equals
 \begin{equation} \label{eqn:Mstudy}
 \bigoplus_{g\in J_\fs\backslash \cK/ \cK_P} \Hom_{J_\fs}\left(\Ind_{J_\fs\cap {}^g\cK_P}^{J_\fs}\circ\Ad(g)\circ\Res^{\cK_P}_{\cK_P\cap J_\fs^g}(\infl^{\cK_P}_{\cK_L}\tau_L ),\lambda_\fs\right).
\end{equation} 
We will consider the term $g=1$ in the sum above. Firstly, we observe, that by \eqref{eqn:intersect}, we have $\Res^{\cK_P}_{\cK_P\cap J_\fs}\circ\infl^{\cK_P}_{\cK_L}=\Res^{\cK_P}_{J_LJ_\fs^+}\circ\infl^{\cK_P}_{\cK_L}$. 
 Whenever $j_P$ belongs to $\cK_P\cap J_\fs=J_LJ_\fs^+$, we have $j_P=j_Lj_U$, with $j_L\in  \cK_L\cap J_\fs=J_L$ and $j_U\in \cK_U\cap J_\fs=J_\fs^+$. Then, we get
\begin{equation} \label{eqn:LHS}
( \Res^{\cK_P}_{\cK_P\cap J_\fs}\circ\infl^{\cK_P}_{\cK_L})\tau_L(j_L)=\tau_L(j_L).
\end{equation}
On the other hand, by \eqref{eqn:intersect}, 
\[
\infl_{\cK_P\cap J_\fs/\cK_P^+\cap J_\fs}^{\cK_P\cap J_\fs}\circ\Res^{\cK_L}_{\cK_P\cap J_\fs/\cK_P^+\cap J_\fs}=\infl_{J_L}^{J_LJ_\fs^+}\circ\Res^{\cK_L}_{J_L},
\] 
and we have
\[(\infl_{J_L}^{J_LJ_\fs^+}\circ\Res^{\cK_L}_{J_L})(\tau_L)(j_L)=\tau_L(j_L).\]
Thus, we  have shown that
\begin{equation} 
\Res^{\cK_P}_{\cK_P\cap J_\fs}\circ\infl^{\cK_P}_{\cK_L}=\infl_{J_L}^{J_LJ_\fs^+}\circ\Res^{\cK_L}_{J_L}.
 \end{equation}
It follows that  the term attached to $g=1$ in \eqref{eqn:Mstudy} equals
\begin{equation}
 \Hom_{J_\fs}\left(\ind_{J_LJ_\fs^+}^{J_\fs}\circ\infl_{J_L}^{J_LJ_\fs^+}\circ\Res^{\cK_L}_{J_L}(\tau_L ),\lambda_\fs\right).
\end{equation}
By Frobenius reciprocity, it equals to 
\begin{equation} \label{eqn:x=1}
\Hom_{J_LJ_\fs^+}\left(\infl_{J_L}^{J_LJ_\fs^+}\circ\Res^{\cK_L}_{J_L}(\tau_L ),\Res_{J_LJ_\fs^+}^{J_\fs}\lambda_\fs\right).
\end{equation}
On the other hand, \eqref{eqn:x=1} equals
\[
\Hom_{J_LJ_\fs^+}\left(\infl_{J_L}^{J_LJ_\fs^+}\circ\Res^{\cK_L}_{J_L}(\tau_L ),\infl_{J_L}^{J_LJ_\fs^+}\lambda_L\right).
\]
since $\Res_{J_L}^{J_\fs}\lambda_\fs=\lambda_L$ and $\lambda_\fs$ is trivial on $J_\fs^+$ (because $(J_\fs,\lambda_\fs)$ is a $G$-cover of $(J_L,\lambda_L)$). On the other hand, by using the Mackey formula we get
\[
\Res^{\cK_L}_{J_L}(\tau_L )=\Res^{\cK_L}_{J_L}(\ind_{J_L}^{\cK_L}\lambda_L )= \bigoplus_{z\in J_L\backslash \cK_L/ J_L} \ind_{J_L\cap J_L^z}^{J_L}\lambda_L^z,
\]
Hence we can deduce that 
\[
\Hom_{J_LJ_\fs^+}\left(\infl_{J_L}^{J_LJ_\fs^+}\circ\Res^{\cK_L}_{J_L}(\tau_L ),\infl_{J_L}^{J_LJ_\fs^+}\lambda_L\right)\neq 0,
\]
by looking at its summand corresponding to $z=1$: 
\[
\Hom_{J_LJ_\fs^+}\left(\infl_{J_L}^{J_LJ_\fs^+}\ind_{J_L}^{J_L}\lambda_L,\infl_{J_L}^{J_LJ_\fs^+}\lambda_L\right)
=\Hom_{J_LJ_\fs^+}\left(\infl_{J_L}^{J_LJ_\fs^+}\lambda_L,\infl_{J_L}^{J_LJ_\fs^+}\lambda_L\right)\neq 0,
\]
and thus the theorem follows.
\end{proof}


\begin{thebibliography}{GKSS}

\bibitem[Ber]{Ber} J. N. Bernstein, ``Le ``centre'' de Bernstein'', in \textit{Representations of reductive groupsover a local field}, Travaux en Cours, Hermann, Paris, 1984, pp. 1--32. Edited by P.~Deligne.

\bibitem[BB]{BB} L. Blasco and C. Blondel, ``Types induits des paraboliques maximaux de $\Sp_4$ et $\GSp_4$'',  Annales de l'Institut Fourier \textbf{49}, no 6 (1999), 1805--1851.

\bibitem[Bl1]{Bl1}  C. Blondel,  ``Crit\`ere d'injectivit\'e pour l’application de Jacquet'', C. R. Acad. Sci. Paris S\'er. I Math. \textbf{325 }(1997), no. 11, 1149--1152.
 

\bibitem[BHS]{BHS} C. Blondel, G. Henniart, and S. Stevens, ``Jordan blocks of cuspidal representations of symplectic groups'', Algebra and Number Theory  (2018).


\bibitem[BT]{BT1} F. Bruhat and  J. Tits, ``Groupes r\'eductifs sur un corps local, I: Donn\'ees radicielles valu\'ees'', Publ. Math. I.H.E.S. \textbf{41} (1972), 1--251.

\bibitem[BH]{BH} C. Bushnell and  G. Henniart, \textit{The local Langlands conjecture for $\GL(2)$}, \textbf{335}, Springer, 2006.

\bibitem[BK1]{BK1} C. Bushnell and P. Kutzko, \textit{The Admissible Dual of $\GL(N)$ via Compact Open Subgroups}, Ann. of Math. Stud., vol. \textbf{129}, Princeton University Press, Princeton, NJ, 1993.
.

\bibitem[BK2]{BK3}  \bysame, ``The admissible dual of $\SL(N)$, II'', Proc. London Math. Soc.  (3) \textbf{68} (1994), no.~2, 317--379

\bibitem[BK3]{BKtyp} \bysame, ``Smooth representations of reductive $p$-adic groups: structure theory via types'', Proc. London Math. Soc. \textbf{77} (1998), no.3,  582--634.


\bibitem[GKS]{GKS} D. Goldberg, P. Kutzko, and S. Stevens, ``Covers for self-dual supercuspidal representations of the Siegel Levi subgroup of classical $p$-adic groups'', Int. Math. Res. Not. IMRN 2007, no. 22, 31 pp.


\bibitem[He]{He} G.~Henniart, ``Sur l'unicité des types pour $\GL_2$'', Appendix to ``Multiplicit\'es modulaires et repr\'esentations de $\GL_2(\Zp)$ et de $\Gal(\overline\QQ_p/\Qp)$ en $l=p$'', by A.~M\'ezard and C.~Breuil, Duke Math. J. \textbf{115} no.~2 (2002), 205--310. 


\bibitem[La1]{La1} P.~Latham, ``Unicity of types for supercuspidal representations of $p$-adic $\SL_2$'', J. Number Theory \textbf{162} (2016), 376--390.

\bibitem[La2]{La2} \bysame, ``The unicity of types for depth-zero supercuspidal representations'', Represent. Theory \textbf{21} (2017), 590--610.

\bibitem[MS]{MS}  M. Miyauchi and S. Stevens, ``Semisimple types for $p$-adic classical groups'', Math. Ann. \textbf{358} (2014), no.~1-2, 257--288. 

\bibitem[MN]{MN} A. K. Mondal and  S.~Nadimpalli, On typical representations for depth-zero components of  split classical groups, Represent. Theory \textbf{23} (2019), 249--277.

\bibitem[MP]{MP} A.  Moy and G. Prasad, Jacquet functors and unrefined minimal $K$-types, 
Comment. Math. Helvetici \textbf{71} (1996) 98-121.

\bibitem[Mo1]{Mor0} L. Morris, ``Tamely ramified supercuspidal representations'', Ann. scien. \'Ec. Norm. Sup. $4^{\mathrm{e}}$ s\'erie, \textbf{29} (1996), 639--667.

\bibitem[Mo2]{Mor} \bysame, ``Level zero $G$-types'', Compositio Math. \textbf{118} (1999), no;.2, 135--157.
 
\bibitem[Na1]{Na1} S.~Nadimpalli, ``Typical representations for level zero Bernstein components of  $\GL_n(F)$'', J. Algebra \textbf{469} (2017), 1--29.

\bibitem[Na2]{Na2} \bysame, ``Typical representations of $\GL_3(F)$'', Forum Mathematicum \textbf{31} no. 4 (2019),  917--941.


\bibitem[Pa]{Pa} V.~Paskunas, ``Unicity of types for supercuspidal representations of $\GL_N$'', Proc. Lond. Math. Soc. (3) \textbf{91} (3) (2005) 623--654.

\end{thebibliography}
\end{document}